\documentclass[reqno, 12pt, final]{amsart}

\usepackage{amssymb,amsfonts,showkeys,graphicx,amsmath,amsthm,amstext,latexsym,amsfonts,mathrsfs}
\usepackage[all]{xy}
\usepackage[latin1]{inputenc}
\usepackage{pstricks,pst-plot}
 \theoremstyle{plain}
 \newtheorem{theorem}{Theorem}[section]
 \newtheorem{definition}[theorem]{Definition}
 \newtheorem{lemma}[theorem]{Lemma}
 \newtheorem{corollary}[theorem]{Corollary}
 \numberwithin{equation}{section} 
 \numberwithin{figure}{section} 
 \newtheorem{prop}[theorem]{Proposition}
 \theoremstyle{remark}
 \newtheorem{remark}[theorem]{Remark}
 
 \newtheorem{example}[theorem]{Example}
\newtheorem*{acknowledgment*}{Acknowledgment}

\def\lip{{\rm Lip}}
\def\eps{\varepsilon}
\def\R{{\mathbb R}} 
\def\N{{\mathbb N}}

\def\M{{\mathcal M}} 
\def\C{{\mathcal C}}

\def\P{{\mathcal P}}

\title[Transport and Disintegrations]{Transport Problems and Disintegration Maps}
\author{Luca Granieri, Francesco Maddalena}
\address{Dipartimento di Matematica Politecnico di Bari, via Orabona 4, 70125 Bari, Italy}
\email{l.granieri@poliba.it, f.maddalena@poliba.it}
\date{}

\begin{document}
\begin{abstract}
By disintegration of transport plans it is introduced the notion of transport class. This allows to consider  the Monge problem   as a particular case of the Kantorovich transport problem, once a transport class is fixed. 
The transport problem constrained to a fixed transport class is equivalent to an
abstract Monge problem over a Wasserstein space of probability measures.
 Concerning solvability of this kind of constrained problems, it turns out that in some sense the Monge problem corresponds to a \emph{lucky} case. 
\end{abstract}

\maketitle
\tableofcontents
\begin{flushleft}
  Calculus of Variations.  Shape Analysis.  Geometric Measure Theory.   
\end{flushleft}

\section*{Introduction}
Optimal transport problems, also known as Monge-Kantorovich problems,
have been very intensively studied in the last  years giving rise  to
 numerous and important applications to PDE, Shape Optimization and
Calculus of Variations, so we witnessed a spectacular development of the
field.
The
interested reader may look at the monographs and  lecture notes
\cite{a1, a-g-s, gangbo,  granieri,pratelli, r-r,   v,v2} where the subject is fully developed.

Let us briefly recall the formulations of the Monge-Kantorovich problems.

Let  $X,Y$  be two compact metric spaces and  let 
$c:X\times Y \to \R^+$ be a Borel cost function.
The Monge problem is formulated as follows: given two probability
measures $\mu \in \P(X)$, $\nu\in \P(Y)$ find a measurable map $t:X \to Y$ such that $t_\#\mu=\nu$  ($\#$ denotes the push-forward of measures) and such
that $t$ minimizes the total cost, i.e. 
\begin{equation}
\label{Monge}
\min_{t:X\rightarrow Y}\left\{\int_X c(x,t(x))\ d \mu,\;\vert\,t_\#\mu=\nu\right\}
\end{equation}

It may happens that the set of
admissible maps is empty (e.g. $\mu=\delta_x$ and $\nu=
\frac{1}{2}(\delta_y+\delta_z)$). Then the problem could be reformulated in
its Kantorovich's relaxation:
find $\gamma \in \P(X \times Y)$ such that 
$\pi ^1 _\# \gamma=\mu$ and $\pi^2 _\# \gamma =\nu$ ($\pi^1$ and $\pi^2$ are the
projections  on the factors of $X \times Y$) and such that $\gamma$ minimizes
the total cost, i.e.
\begin{equation}
\label{Monge-Kantorovich}
\min_{\gamma}\left\{\int_{X\times Y} c(x,y)\ d\gamma(x,y)\;\vert\;\pi ^1 _\# \gamma=\mu,\:\pi^2 _\# \gamma =\nu\right\}.
\end{equation} 
The admissible measures $\gamma$
for the Kantorovich problem are called transport plans. We denote by $\Pi(\mu,\nu)$ the set of transport plans with marginals $\mu$ and $\nu$.
If $t$ is admissible for the Monge problem then the measure associated
in the usual way to the graph of $t$, i.e. $\gamma=(I_X\times t)_ \# \mu$,  is admissible for the Kantorovich
problem. However the class of admissible measures for the Kantorovich
problem is never empty as it contains $\mu \otimes \nu$. Moreover, the
Kantorovich problem is a linear one. 
Existence of minimizers for the Monge problem is difficult and may
fails, while for the Kantorovich problem the semicontinuity of $c$ is
enough to ensure existence of minimizers.

If $X=Y$, and $c=d$ is  the distance function,  for $ p\geq1$ the cost 
$$W_p(\mu,\nu)= \left ( \min
\left \{ \int _{X \times X}d^p(x,y)\ d
\gamma(x,y) \ : \  \gamma \in  \Pi(\mu,\nu) \right \} \right )^{1/p} $$
defines a distance on $\P(M)$ called $p$-Wasserstein distance.



Let us recall that by Kantorovich duality (see \cite{a1,granieri, v,v2}) the $1$-Wasserstein distance between $\mu$ and $\nu$, which we will simply denote by $W$,  can be expressed as follows\begin{equation}\label{wd1}W(\mu,\nu)=\sup\left\{\int_X\varphi\;d(\mu-\nu)\;\vert\; \varphi\in\lip_1(X)\right\},\end{equation} where $\lip_1(X)$ denotes the set of Lipschitz function having Lipschitz constant not greater than one.
\subsection*{Description of the results}
A relevant  tool in mass transportation theory  is constituted by the Disintegration Theorem (Theorem~\ref{disintegration}) of measures which states that 
 every transport plan $\gamma \in \P(X\times Y)$  can be written as 
 $\gamma =f(x)\otimes \mu$ where $f(x)\in \P(Y)$. 
 We shall call {\em disintegration map} every $f:X\rightarrow (\P(Y), W)$ such that
$$f(x)\otimes \mu \in \Pi (\mu,\nu).$$

In this paper we relate the structure of the set of transport plans $\Pi(\mu,\nu)$ with the push-forward of disintegration maps. Indeed, given the measure $\gamma=f(x)\otimes \mu$, obviously $\mu$ is the first marginal of $\gamma$, 
 while the second marginal depends on the disintegration map $f$. 
 An interesting feature of transport plans appears  by looking to the measure 
 $f_\# \mu $. Precisely, if  $\eta=g(x)\otimes \mu$ is another transport plan, 
 it results  (see Lemma \ref{push-lemma}):
$$ f_\# \mu = g_\# \mu\Rightarrow \pi^2_\# \gamma =\pi^2_\# \eta .$$ Therefore, the second marginals can be fixed by looking to the push-forward of disintegration maps. In this way the set of transport plans $\Pi(\mu,\nu)$ can be structured in transport classes  (see Definition \ref{class}) by setting $\eta \in [\gamma] \Leftrightarrow f_\# \mu = g_\# \mu$. Roughly speaking 
(see Example~\ref{discrete-class}), fixing a transport class leads  to consider a constrained transport problem with respect to splitting masses or traveling ones.  
  Lemma~\ref{push-metric} shows that all transport plans induced by transport maps belong to the same transport class. Moreover, by density of transport maps in $\Pi(\mu,\nu )$, it follows that  (see Proposition~\ref{prop-class}) such transport class characterizes the transport plans induced by transport maps.\\ Therefore, in this perspective the Monge problem can be seen as a constrained Kantorovich problem, namely 
$${\rm min} \left \{\int_X c(x,t(x))\ d\mu \ :\  t_\#\mu = \nu\right \} ={\rm min}\left \{\int_{X\times Y}c(x,y)\ d \gamma \ : \ \gamma \in [\delta_s\otimes \mu] \right \},$$ 
for a given transport map $s$.  By density of transport maps, the Kantorovich transport problem also  corresponds to 
$${\rm min} \left \{\int_{X\times Y} c(x,y)\  d\gamma  \ :\  \gamma \in \Pi (\mu, \nu)\right \} =  
{\rm min}\left \{\int_{X\times Y}c(x,y)\ d \gamma \ : \ \gamma \in \overline{[\delta_s\otimes \mu]}^W \right \},$$ 
for a given  transport map $s$.\\
Hence  the Monge   problem represents  a particular case of a more rich structure of problems, obtained by fixing a transport class in the Kantorovich formulation of transport problems. In this context, considering transport problems in a fixed transport class is quite natural. 
Since we are looking to the push-forward trough disintegration maps, fixing a transport class results equivalent (see Section $2$) to consider measures $\Lambda \in \P(\P(Y))$ satisfying the \emph{barycenter} constraint 
$$\int_{\P(Y)} \lambda \ d\Lambda (\lambda) = \nu .$$ The corresponding transport class is given by the transport plans $f(x)\otimes \mu$ such that $f_\#\mu = \Lambda$. 
In  this formulation  we see that fixing a transport class is equivalent to consider transport maps $f$ sending $\mu$ into $\Lambda$.
So it is  natural to consider the following Monge-Kantorovich problem in the class $\Lambda$:  
\begin{equation}
  MK_\Lambda(c,\mu,\nu):= \inf _\gamma  \left \{ \int_{X\times Y}c(x,y)\ d\gamma\;\vert\;  \gamma =f\otimes \mu,\;  f_\#\mu=\Lambda \right \}.
 \end{equation} 
The above transport problem leads  to consider an abstract Monge problem between the space $X$ and $\P(Y)$. 
Let us consider the following transport cost  
$$ \forall (x,\lambda) \in X\times \P(Y) \ : \ \tilde c(x,\lambda) = \int_Y c(x,y)\ d\lambda .$$
We set
\begin{equation}
M(\tilde c, \mu, \Lambda):=\inf_{f}\left\{\int_X\tilde c(x,f(x))\ d\mu\;\vert\;f_\#\mu=\Lambda  \right\}.
\end{equation}
For every transport class $\Lambda$ (see Proposition~\ref{MKLambda})
 we have
$$ M(\tilde c, \mu, \Lambda)=MK_\Lambda(c,\mu,\nu).$$
Therefore, every existence result for the Monge problem $M(\tilde c, \mu, \Lambda)$ in the abstract setting corresponds to an existence result for the Monge-Kantorovich problem in the transport class $\Lambda$. Of course, minimizing in a transport class could be as difficult as for the Monge problem since, of course, the transport classes are not in general closed. However, by this reformulation it comes out that the Monge case is peculiar. More precisely, since the Monge problem is a particular case of transportation in a transport class, one may asks what happens for others transport classes.
In other words, the matter consists in establishing if the abstract Monge problem admits solutions. The existence results for the Monge problem are usually stated in the following form: under some assumption on the spaces, on the first marginal $\mu$ and on the cost $c(x,y)$,  for every second marginal $\nu$ the Monge problem admits solutions. For the abstract Monge problem $ M(\tilde c, \mu, \Lambda)$ this is not the case. For discrete measures $\Lambda$, see Section $4$,  for the quadratic cost it results that $ M(\tilde c, \mu, \Lambda)$ may not admit solutions. Therefore, in the abstract setting, it could be also interesting to consider, under some assumption on the spaces, on the first marginal $\mu$ and on the cost $c(x,y)$,  the question for what kind of second marginals the corresponding Monge problem admits solutions. From this point of view, in some sense the Monge problem is a \emph{lucky} case.


\section{Disintegration maps and transport classes}
Let $X,Y\subset \R^N$ be two compact sets and 
let $\M(Y)$ be the space of Radon measures on $Y$. 
A map $\lambda :X\rightarrow \M (Y)$ is said to be Borel, or equivalently weakly*-measurable,  if for any open set 
$B\subset Y$ the function $x\in X \mapsto \lambda_x (B)$ is a real valued Borel map.
Equivalently, $x\mapsto \lambda_x$ is a Borel map if, for any Borel and bounded map $ \varphi :
X\times Y \rightarrow \R$, it results that the map $$ x \in X
\mapsto \int_Y \varphi (x,y) d\lambda _x $$ is Borel.
\begin{theorem}[Disintegration theorem]
\label{disintegration} Let $\gamma  \in \P (X\times Y ) $ be given and let
$ \pi^1 : X\times Y  \rightarrow X$ be the first projection map of $X\times Y$, we set $\mu = (\pi^1)_ \# \gamma$.  Then for $\mu-a.e.\ x \in X $ there exists $\nu _x \in \P
(Y)$   such that
\begin{itemize}
\item[\rm (i)] the map $x \mapsto \nu _x$ is Borel,\\
\item[\rm (ii)]$ \forall \varphi \in \C _b (X\times Y ) : \, \int _{X\times Y} \varphi (x,y) d\gamma
= \int _X \left ( \int _Y \varphi (x,y) d\nu _x(y)
\right )d\mu (x) .$
\end{itemize}
Moreover the measures $\nu _x$ are uniquely determined up to a negligible set with respect to  $\mu$. 
\end{theorem}
Let $\gamma\in\Pi(\mu,\nu)$,
as usual we will  write $ \gamma =  \nu _x\otimes \mu $, 
 assuming that $ \nu _x $
satisfy the condition (i) and (ii) of Theorem~\ref{disintegration}. Obviously, the transport plan $\mu \otimes \nu$ corresponds to the constant map $x\mapsto\nu_x=\nu$. 
Let $t: X\rightarrow Y$, be a transport map, 
observe that for the  transport plan $\gamma _t:= (I\times t)_ \# \mu$, the  Disintegration Theorem yields $\gamma _t=  \delta _{t(x)}\otimes \mu $.
Therefore, the disintegration procedure for a transport plan $\gamma = f(x)\otimes \mu$ produces  a map
\begin{equation}
f:X\rightarrow (\P(Y),W),\:\:\hbox{s.t.}\:\:  x\mapsto f(x)\:\: \hbox{is Borel}.
\end{equation}
We shall refer to such maps $f$  as {\em disintegration maps}. For a transport map $t$ the corresponding  disintegration map is given by $x\mapsto \delta_{t(x)}$.
Of course, it is possible to look at a disintegration map as a measurable map between $X$ and $(\P(Y),W)$. Indeed, we have the following result.
\begin{lemma}\label{measurable}
A map $f:X\rightarrow (\P(Y),W)$ is a disintegration map if and only if $f$ is measurable.
\end{lemma}
\begin{proof}
Let $f:X\rightarrow (\P(Y),W)$ be measurable and let  $A\subset Y$ be an open set. Observe that $f(x)(A)=\int _Y \chi_A(y) df(x)$. For a l.s.c. function $\varphi$ over $Y$, define
 
\begin{equation}\label{valutation}I_\varphi : (\P(Y),W) \rightarrow \R, \quad  I_\varphi(\lambda):=\int_{\P(Y)}\varphi (y) d\lambda .\end{equation}
Since $W$ metrizes the weak* topology of measures, we have that $I_\varphi$ is a l.s.c. map.
For every $x\in X$ it results 
$$\int_Y\varphi (y) df(x)= I_\varphi(f(x)).$$
If $f$ is a measurable map, it follows that the map $f(\cdot)(A):X\rightarrow \R $ is Borel as composition of a l.s.c map and a measurable one. Hence $f$ is a disintegration map.
Vice versa, observe that by the Ascoli-Arzel\'{a} Theorem the space $\lip_1(Y)$ is compact with respect to the uniform convergence. Fixed a countable dense subset $D\subset \lip_1(Y)$, by Kantorovich duality we have
$$ W(\nu_1,\nu_2)=\sup_{u\in \lip_1(Y)}\int_Y u\ d(\nu_1-\nu_2)=\sup_{u\in D}\int_Y u\ d(\nu_1-\nu_2).$$
Since $x\mapsto f(x)$ is Borel, we have that, for every $u\in  \lip_1(Y)$,
  $g:X\rightarrow \R$ defined by $g_u(x):= \int_Y u\ d(\nu-f(x))$ is a Borel map .
To check that $f$ is a measurable map, it is sufficient to observe that $$f^{-1}(B(\nu,r))= \bigcap _{u\in D}g_u^{-1}(]-r,r[):=A.$$
Indeed, if $x\in A$ we get $|g_u(x)|<r$, for every $u\in D$. Hence, by definition of $g_u$, it follows that $W(\nu,f(x))<r$ and then $f(x)\in B(\nu,r)$.
On the other hand, if $f(x)\in B(\nu,r)$, i.e. $W(\nu,f(x))<r$, by Kantorovich duality $|g_u(x)|:= |\int_Y u\ d(\nu-f(x))|<r$ for every $u\in D$. This implies $x\in A$. 
\end{proof}
 Let  $X\subset \R ^N$,  we recall that the  barycenter of a measure $\mu\in \P(X)$ is given by 
$$\beta (\mu )= \int _{X}x\ d\mu .$$

  Disintegration maps naturally produce   measures of the form  $f_\# \mu$ on  the space $(\P(Y),W)$. By the following lemma,  we see that this point of view is equivalent to fix the second marginal  of transport plans induced by transport maps.


\begin{lemma}
\label{push-metric}
Let $t,s:X\rightarrow Y$ be two  given Borel maps, $\mu\in\P(X)$ and let $f,g:X\rightarrow \P(Y)$ defined by    $f(x)=\delta_{t(x)}$, $g(x)=\delta_{s(x)}$.  Then
 \begin{equation}
   t_\# \mu = s_\# \mu\:\Leftrightarrow\: f_\# \mu = g_\# \mu.
\end{equation}
\end{lemma}
\begin{proof}
Assume  $f_\# \mu = g_\# \mu$ and  for any  $\varphi \in C(Y)$ let us consider  the  function $I_\varphi$ defined in \eqref{valutation}.
Observe that $I_\varphi \in C ((\P(Y),W))$. Hence we have 
$$ \int_{\P(Y)} I_\varphi(\lambda) \ d(f_\# \mu)=\int_{\P(Y)} I_\varphi(\lambda) \ d(g_\# \mu) \Leftrightarrow \int _XI_\varphi(f(x))\ d\mu = \int _X I_\varphi (g(x))\ d \mu$$
\begin{equation}\label{marginal} \Leftrightarrow \int_X \left ( \int _Y \varphi (y)\ df(x) \right )  d\mu = \int_X \left ( \int _Y \varphi (y)\ dg(x) \right )  d\mu.\end{equation}
Since
$$ \int_Y\varphi(y)\ df(x)=\varphi(t(x)),\:\:\:\:\int_Y\varphi(y)\ dg(x)=\varphi(s(x)),$$
bi \eqref{marginal} we get 
$$\int_X\varphi(t(x))\ d\mu=\int_X\varphi(s(x))\ d\mu .$$
By the arbitrariness of $\varphi$ we infer  $t_\#\mu=s_\#\mu$.

Vice-versa, for every  $\psi \in C((\P(Y),W))$ let us consider the function $\varphi(y)=\psi(\delta_y)$. Observe that $\varphi \in C(Y)$. 
If  $t_\# \mu = s_\# \mu$ we compute
$$ \int_Y \psi \ d(f_\# \mu)= \int_X \psi(f(x))\ d\mu =\int_X\psi (\delta_{t(x)})\ d\mu =$$
$$ \int_X \varphi (t(x))\ d\mu = \int_X \varphi (s(x))\ d\mu =
\int_Y \psi\  d(g_\# \mu).$$
By the arbitrariness of $\psi $ we obtain  $f_\# \mu = g_\# \mu$.
\end{proof}

\begin{corollary}
Let $t,s:X\rightarrow Y$ be two  given Borel maps, $\mu\in\P(X)$,   let $f,g:X\rightarrow \P(Y)$ defined by    $f(x)=\delta_{t(x)}$, $g(x)=\delta_{s(x)}$ and let $\gamma = f(x)\otimes \mu$, 
$\eta =  g(x)\otimes \mu$. Then 
\begin{equation} 
\pi^2_\# \gamma =\pi^2_\# \eta \Leftrightarrow f_\# \mu = g_\#\mu.
\end{equation}
\end{corollary}

Observe that the first part of the proof of the above Lemma works for general transport plans $\gamma = f(x)\otimes \mu$, 
$\eta =  g(x)\otimes \mu$. Actually, by \eqref{marginal} we get  the following

\begin{lemma}\label{push-lemma}
Let $\mu\in\P(X)$,    $f,g:X\rightarrow \P(Y)$,  $\gamma =f(x)\otimes \mu $, $\eta =  g(x)\otimes \mu $ be given. Then the following implication holds true
  \begin{equation}
 \label{push-implication}
   f_\# \mu = g_\# \mu\Rightarrow \pi^2_\# \gamma =\pi^2_\# \eta .
\end{equation}
\end{lemma}
Therefore, also for transport plans, the second marginal can be fixed by fixing the push-forward of disintegration maps.

Notice that  in general  the converse of \eqref{push-implication}  is not  true as we show in the next example.
\begin{example}
Let $f:X\rightarrow \P(Y)$ defined  by $f(x)=\nu$ and let 
$\gamma= f(x)\otimes \mu $.  Let $\eta=g(x)\otimes \mu $ where $g(x)=\delta_{t(x)}$ for a  given transport map $t:X\rightarrow Y$ with  $t_\# \mu =\nu$. 
For every $\psi \in \C((\P(Y),W))$ we have 
$$\int_{\P(Y)}\psi \ d(g_\# \mu)=\int_X\psi (\delta_{t(x)})\ d\mu,$$ 
while
$$\int_{\P(Y)}\psi \ d(f_\# \mu)=\psi(\nu).$$ 
For any  $\varphi\in \C(Y)$ let us consider $\psi (\lambda)=\left | \int_Y \varphi (y) d \lambda \right |=|I_\varphi (\lambda)  |$. We compute
$$ \int_X\psi (\delta_{t(x)})\ d\mu =\int_X \left | \int_Y \varphi (y)\ d \delta_{t(x)}\right |\ d\mu = \int_X|\varphi(t(x))|\ d\mu = \int _Y|\varphi (y)|\ d\nu.$$ However, on the other hand $\psi(\nu)=\left | \int_Y \varphi (y)\  d\nu \right |$.
\end{example}
The above arguments allow to characterize transport plans through the 
push forward of disintegration maps. We introduce the following notion of transport class.
\begin{definition}[Transport classes]\label{class}
Let $\gamma, \eta\in \Pi(\mu,\nu)$ with $\gamma= f(x)\otimes \mu$, $\eta= g(x)\otimes \mu $ be given. We shall say that $\gamma$ and $\eta$ are equivalent (by disintegration), in symbols $\gamma\approx\eta$, if $f_\#\mu= g_\#\mu$.\\ For any given $\eta\in \Pi(\mu,\nu)$ with $\eta= g(x)\otimes \mu $,  we shall call transport class of $\eta$ the   equivalence class  of  transport plans given by
\begin{equation}
\label{equiv-class}
[\eta]=\{\gamma= f(x)\otimes \mu\;\vert\;f_\#\mu= g_\#\mu\}.
\end{equation}
\end{definition}
 Notice that in the case of  discrete first marginal $\mu =\sum_i \alpha_i \delta_{x_i}$, for any disintegration map it is easily seen that 
 $$ f_\#\mu= \sum_i \alpha_i \delta_{f(x_i)}.$$ Therefore, transport classes are fixed by the range of $f$. 
 
 \begin{example}\label{discrete-class}
 Let  $$\mu=\frac{1}{3}\delta_{x_1}+\frac{1}{3}\delta_{x_2}+\frac{1}{3}\delta_{x_3}, \quad \nu =\frac{1}{6}\delta_{y_1}+\frac{5}{6}\delta_{y_2}.$$
 
\begin{figure}[htbp]
	\centering
		\includegraphics{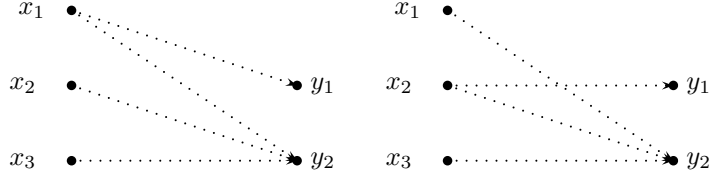}
	\caption{Transport plans in the same class.}
	\label{fig:discrete-class}
\end{figure}

 Consider the transport plan which uniquely splits the mass at $x_1$. This transport plan corresponds to the disintegration map
 \begin{equation}\label{e1}f(x_1)=3(a\delta_{y_1}+b\delta_{y_2}), \quad f(x_2)=\delta_{y_2}, \quad f(x_3)=\delta_{y_2}, \quad a=b=\frac{1}{6}. \end{equation} By changing the point at which the mass is splitted,  the range of the corresponding disintegration map does not change.
 For instance, for the second transport plan in Figure \ref{fig:discrete-class} we get the following disintegration map
 $$ g(x_1)=\delta _{y_2},\quad  g(x_2)= 3(a\delta_{y_1}+b\delta_{y_2}),  \quad g(x_3)=\delta_{y_2}, \quad a=b=\frac{1}{6}.$$ It follows  $f_\# \mu=g_\#\mu$. 
  Analogously, all the transport plans with only one splitted mass  belong to the same transport class. \\
  On the other hand, by changing the number of splitted masses the corresponding disintegration range is changing. 
  
\begin{figure}[htbp]
	\centering
		\includegraphics{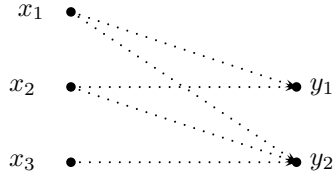}
	\caption{Two splitting masses. }
	\label{fig:discrete-class2}
\end{figure}

  Indeed, by looking at Figure \ref{fig:discrete-class2}, we may  consider  the disintegration map
 $$ 
 h(x_1)=3(a'\delta_{y_1}+b'\delta_{y_2}), \quad h(x_2)=3(c'\delta_{y_1}+d'\delta_{y_2}), \quad h(x_3)=\delta_{y_2}, $$ $$  a'=\frac{3}{30}, \  b'=\frac{7}{30}, \ c'=\frac{2}{30}, \  d'=\frac{8}{30}. 
 $$
 In such a case we have that $f_\# \mu \neq h_\#\mu.$ On the other hand, by keeping fixed the number of splitted masses, the transport class may be changed by modifying the amount of traveling masses. Consider for instance the disintegration
  $$ 
 k(x_1)=3(a''\delta_{y_1}+b''\delta_{y_2}), \quad k(x_2)=3(c''\delta_{y_1}+d''\delta_{y_2}), \quad k(x_3)=\delta_{y_2}, $$ $$  a''=\frac{1}{30}, \  b''=\frac{9}{30}, \ c''=\frac{4}{30}, \  d''=\frac{6}{30}. 
 $$
 We get $h_\# \mu \neq k_\#\mu.$
 
 Therefore, to fix a transport class leads to consider a {\em constrained transport problem}, with respect to splitting masses or traveling  ones.
 \end{example}


By Lemma \ref{push-metric} it follows that all transport plans induced by transport maps belong to the same transport class. Since the transport maps are dense in $\Pi(\mu,\nu)$  we can prove the following result.

\begin{prop}\label{prop-class}
Let $s:X\rightarrow Y$,  be a transport map, i.e. such that $s_\#\mu=\nu$, with $\mu$ non-atomic,  and let 
$\eta =(I\times s)_\# \mu= \delta _{s(x)}\otimes \mu$.
If $\gamma\in[\eta]$ then there exists a transport map $t:X\rightarrow Y$ such that
$\gamma=\delta _{t(x)}\otimes \mu $, i.e. the transport plan $\gamma$ is induced by a transport map  $t$. In particular, if $\gamma =f(x)\otimes \mu$, it results $t(x)=\beta(f(x))$   $\mu$-a.e..
\end{prop}
\begin{proof}
By applying \cite[Theorem 9.3]{a1}, see also \cite{gangbo},  and the same argument employed in the proof of  (\cite[Theorem 2.1]{a1}), we find   a sequence of Borel maps $t_n:X\rightarrow Y$ such that 
$$\gamma =\lim_{n\rightarrow + \infty}  \delta _{t_n(x)}\otimes \mu ,\:\:\:(t_n)_\# \mu =\nu\:\:\:\forall n\in \N$$ 
and therefore
 $ \delta _{t_n(x)}\otimes \mu \in [\eta]$,  $\forall n\in\N$.
 Consider $\varphi(y)=|y|^2.$ By the push-forward constraint we get
 $$ \int_X|t_n(x)|^2\ d\mu = \int_X\varphi(t_n(x))\ d\mu= \int_Y\varphi(y)\ d\nu=\int _Y|y|^2 d\nu <+\infty .$$
 
Let us consider  $\psi\in \C((\P(Y),W))$ defined  by  
$\psi (\delta_y)=|y|^2$. In fact, setting $\Delta \subset \P(Y)$ the set of Dirac deltas, the function $\psi$ is Lipschitz, with respect to the Wasserstein distance, over $\Delta$. Hence it suffices to consider any Lipschitz extension of $\psi$  on the whole $\P(Y)$.
For every $n\in \N$, since $(\delta_{t_n})_\#\mu=(\delta_s)_\# \mu $ we have 
\begin{equation}\label{l2} \int_X |t_n(x)|^2 d\mu =  \int _X \psi(\delta_{t_n(x)})\ d\mu = \int _X \psi(\delta_{s(x)})\ d\mu=\int_X|s|^2 d\mu.\end{equation}
Therefore, by passing to a subsequence, we may suppose that   $t_n$ is  weakly convergent and let   $t$ be the weak limit of $t_n$.\\
 Let $\gamma =f(x)\otimes \mu$. By  definition of weak convergence, by approximation with continuous functions,  for any $g\in L^2(X, \R^N)$ 
  we get
$$\int_X \langle g,t\rangle d\mu = \lim_{n\rightarrow + \infty}\int_X\langle g,t_n\rangle d\mu =
\lim_{n\rightarrow + \infty}\int_X \left ( \int_Y \langle g,y\rangle d \delta_{t_n(x)}(y)\right) d\mu =$$ $$\int_X \left ( \int_Y \langle g,y\rangle \ d f(x)\right) d\mu = \int_X \left \langle g, \int_Y y\  df(x)\right \rangle  d\mu= \int_X\langle g, \beta (f(x))\rangle d\mu .$$
Therefore  $t_n\rightharpoonup\beta (f(x))$. 
On the other hand, since  $\gamma \in [\eta ]$, i.e. $f_\#\mu=(\delta_{t_n})_\#\mu$,  and by \eqref{l2} we have
 
$$ \int_X |\beta (f(x))|^2 d\mu= \int_X  |\beta (\delta _{s(x)})|^2 d\mu = \int_X |s|^2 d\mu = \int_X |t_n |^2 d\mu .$$
Then, it follows that $t_n$ strongly converges to $\beta (f(x))$, hence by
\cite[Lemma 9.1]{a1} we deduce that   $\gamma =\delta_{t(x)}\otimes \mu $ with $t(x)=\beta (f(x))$.
\end{proof}

Of course, the above arguments  hold true  as well by considering $L^p(X, \mu)$
with $p>1$.  This corresponds to consider the transport cost $d^p(x,y)$.\\
 Proposition~\ref{prop-class} allows to reformulate the  Monge transport  problem  as follows:
\begin{equation}\label{monge-class}{\rm min} \left \{\int_X c(x,t(x))\ d\mu \ :\  t_\#\mu = \nu\right \} = {\rm min}\left \{\int_{X\times Y}c(x,y)\ d \gamma \ : \ \gamma \in [\delta_s\otimes \mu] \right \},\end{equation} for a 
given transport map $s$.\\ 
By density of transport maps, the Kantorovich transport problem can be seen as
\begin{equation}\label{kant-class}{\rm min} \left \{\int_{X\times Y} c(x,y))\ d\gamma  \ :\  \gamma \in \Pi (\mu, \nu)\right \} = {\rm min}\left \{\int_{X\times Y}c(x,y)\ d \gamma \ : \ \gamma \in \overline{[\delta_t\otimes \mu]}^W \right \},\end{equation}for a given transport map $t$.
 
 Therefore, the  Monge problem corresponds to minimize  the functional 
 $\int_{X\times Y} c(x,y)\ d\gamma$ in a fixed transport class of $\Pi(\mu,\nu)$, while the Kantorovich one corresponds to  minimize the same functional on the whole $\Pi(\mu,\nu)$.
\section{Monge-Kantorovich problems on transport classes}
In the previous section we have seen that the Monge problem could be seen as a particular case of minimization on a transport class. Since the transport classes are determined through  the push-forward of disintegration maps, they can be assigned by  probability measures $\Lambda $ over $(\P(Y),W)$.\\
 Actually, consider $f\otimes \mu \in \Pi(\mu,\nu)$ and $\Lambda =f_\#\mu$. 
Since $(\pi_2)_\#( f\otimes \mu)=\nu$, for every $\varphi \in \C(Y)$ we have
$$\int_Y\varphi(y)\ d\nu = \int_X\left (\int_Y\ \varphi (y) \ df(x)\right ) d\mu = \int_X I_\varphi(f(x))\ d\mu = \int_{\P(Y)}I_\varphi(\lambda ) \ d\Lambda(\lambda)=
$$
$$= \int_{\P(Y)}\left ( \int_Y \varphi (y) \ d\lambda\right ) d\Lambda .$$
Therefore, in order to define a transport class,  the measure $\Lambda $ has to satisfy the constraint  
\begin{equation}\label{lambda-class}\int_{\P(Y)}\lambda \ d\Lambda =\nu .\end{equation}
Hence, every probability measure $\Lambda$ over $(\P(Y),W)$ satisfying \eqref{lambda-class} defines a transport class $[\eta]=\{f\otimes \mu : f_\#\mu =\Lambda \}$.\\ 
For instance, the transport class $[\mu\otimes \nu]$ corresponds to the measure $\Lambda =\delta_\nu$. While for a transport map $t$, the transport class
$[\delta_{t(x)}\otimes \mu]$ corresponds to $\Lambda = \int_X \delta_{\delta_{t(x)}} d\mu$. On the other hand,  the transport class in \eqref{e1} corresponds to the discrete measure $\Lambda = \frac{1}{6}\delta_{\delta_{y_1}}+\frac{5}{6}\delta_{\delta_{y_2}}$. In this perspective, transport plans in the transport class $\Lambda$ can be seen as transport maps between $\mu$ and $\Lambda $. It is then natural to consider the Monge-Kantorovich problem in the class $\Lambda$  defined as follows
\begin{equation}\label{monge-class} MK_\Lambda(c,\mu,\nu):= \inf _\gamma  \left \{ \int_{X\times Y}c(x,y)\ d\gamma \;\vert\;   \gamma =f\otimes \mu, f_\#\mu=\Lambda \right \} \end{equation}
By Proposition~\ref{prop-class}, the Monge problem corresponds to the transport class $\Lambda = \int_X \delta_{\delta_{t(x)}} d\mu$.
\begin{remark}
{\rm Observe that for a discrete transport class $\Lambda =\sum_i \alpha_i \delta_{\lambda_i}$, the Monge-Kantorovich problem in the class $\Lambda$ yields the {\em optimal allocation problem} consisting in the determination of   a Borel subdivision $\{U_i\}$ of $X$ minimizing 
$$ \sum_i \int_{U_i} \left ( \int_Yc(x,y) d\lambda _i \right ) d\mu , $$ under  the constraint  $\sum_i\alpha_i \lambda_i =\nu$. For existence results in this framework see Section $4$.} 
\end{remark}

The notion of transport class leads naturally to consider an abstract Monge problem between the space $X$ and $\P(Y)$. Consider the following transport cost
\begin{equation}\label{ctilde} \forall (x,\lambda) \in X\times \P(Y) \ : \ \tilde c(x,\lambda) = \int_Y c(x,y) d\lambda .\end{equation}
We have the following
\begin{prop}
\label{MKLambda}
For every transport class $\Lambda\subset \Pi(\mu,\nu)$ we have 
$$ M(\tilde c, \mu, \Lambda)=MK_\Lambda(c,\mu,\nu).$$
\end{prop}
\begin{proof}
It suffices to observe that for any disintegration map $f:X\rightarrow \P(Y)$ such that $f_\#\mu = \Lambda$, it results
$$ \int_X \tilde c (x,f(x))\ d\mu = \int _X \left ( \int_Y c(x,y) df(x) \right ) d\mu = \int_{X\times Y} c(x,y)\ d(f\otimes \mu).$$
\end{proof}
Observe that by the above proof it follows that  $f$ is a solution of $M(\tilde c, \mu, \Lambda)$ if and only if $f\otimes \mu$ is a solution of 
$MK_\Lambda(c,\mu,\nu)$. 
Therefore, every existence result for the Monge problem $M(\tilde c, \mu, \Lambda)$ in the abstract setting corresponds to an existence result for the Monge-Kantorovich problem in the transport class $\Lambda$. The abstract setting has the advantage of considering a \emph{nice} cost,  since it is linear with respect to  the second variable. Of course the disadvantage is the passage from the space $X\times Y\subset \R^N\times \R^N$ to the space $X\times \P(Y)$. Observe that by \cite[Theorem 2.4]{pratelli} the transport classes are always non-empty provided $\mu$ is non-atomic. 
\begin{prop}\label{cost-reg}
$c$ is continuous, l.s.c., Caratheodory, normal iff $ \tilde c $ is.
\end{prop}
Where we say that a measurable map $c(x,y)$ is Caratheodory (resp. normal) if $c(x,\cdot)$ is continuous (resp. l.s.c.).
\begin{lemma}\label{reg}
Let $c : X\times Y \rightarrow [0,+\infty]$ be a Borel cost function satisfying
\begin{equation}\label{reg-cond}
|c(x_1,y)-c(x_2,y)|\leq \alpha (x_1-x_2), \end{equation}
for a given map $\alpha : X \rightarrow \R $ continuous at $x=0$ and such that $\alpha(0)=0$. 
We have the following
\begin{enumerate}
\item If $c(x,\cdot)$ is continuous then $ c$ (and hence $\tilde c$) is continuous.
\item If $c(x,\cdot)$ is l.s.c. then $c$ (and hence $\tilde c$) is l.s.c.
\end{enumerate}
\end{lemma}
\begin{proof}
Let $c(x,\cdot ) $ be continuous. 
If $(x_n, y_n)\rightarrow (x,y)$ on $X\times Y$ we compute
$$ | c(x,y)- c(x_n, y _n)| \leq | c(x,y)- c(x,y _n)|+| c(x,y _n)- c(x_n,y _n)|\leq$$
$$ | c(x,y)- c(x,y _n)|+ 
\alpha (x-x_n) \rightarrow 0 ,$$ as $n\rightarrow + \infty$. 
If $c(x,\cdot )$ is l.s.c. 
considering $$c(x_n,y_n) =  c(x_n,y_n)- c(x,y_n)+ c(x,y_n),$$ by \eqref{reg-cond}, passing to the liminf we obtain
$$\liminf_{n\rightarrow + \infty} c(x_n,y_n)= \liminf_{n\rightarrow + \infty} c(x,y_n)\geq  c(x,y).$$
\end{proof}
In general, for the existence of optimal transport plans in the Kantorovich problem at least the lower semicontinuity  property of the cost function is usually required. 
Actually, some regularity of the cost function is needed to obtain a useful duality theory or to ensure that the infimum of the Kantorovich problem is equal to the infimum of the Monge problem  (see for instance \cite{a1, pratelli}). However, it is not hard to verify that the Kantorovich problem admits solutions under more weak requirements  on the cost function  (see for instance \cite{bernard}). 

For  reader's convenience  here we provide  some details based on disintegration maps.\\ 
Let $\pi_n=f_n\otimes \mu $ be a minimizing sequence for the Monge-Kantorovich problem \eqref{Monge-Kantorovich}. Passing to a subsequence we may suppose that $\pi_n\rightharpoonup f\otimes \mu =\pi \in \Pi(\mu,\nu)$. Now, for any  $\psi(x)\in \C(X), \varphi (y)\in \C(Y)$, we get
\begin{equation}\label{weak-conv}\begin{split} \int_{X\times Y}\psi(x)\varphi (y)\ d \pi =  \int_X \psi (x) \left ( \int_Y \varphi (y) df(x) \right ) d\mu =\\  \lim_{n\rightarrow + \infty}\int_X \psi (x)\left ( \int_Y \varphi (y) df _n(x) \right ) d\mu =\lim_{n\rightarrow + \infty}\int_{X\times Y}\psi(x)\varphi (y)\ d \pi _n .\end{split}\end{equation}
By density of continuous functions, the above limit holds for $\psi \in L^1(X,\mu)$ as well. Therefore, if the cost function has  the form
$c(x,y)=a(x)b(y)$, with $a\in L^1, b\in \C$,  then by \eqref{weak-conv} it follows that $\pi $ is an optimal transport plan. 
Arguing component-wise, the same reasonings apply to linear costs $c(x,y)=\langle a(x), y \rangle $. For a Caratheodory cost function $c(x,y)$, i.e. a Borel map such that $c(x,\cdot )$ is continuous,  observe that the disintegration maps $f_n:X\rightarrow \P(Y)\subset \mathcal M(Y,\R)$ belongs to $L^\infty (X, \mathcal M(Y,\R))$, which is the dual of $L^1(X, \C(Y))$. Therefore, by passing to a subsequence we may suppose that $f_n \stackrel{*}{\rightharpoonup}f$, i.e. 
$$ \lim_{n\rightarrow + \infty} \int_ X \left ( \int_Y \psi(x,y) df_n(x) \right ) d\mu = \int_ X \left ( \int_Y \psi(x,y) df (x) \right )d\mu\:\:\:
\forall \psi \in L^1(X, \C(Y)).$$ 
The above continuity property shows that $\pi=f\otimes \mu$ is an optimal transport plan, 
provided that $\int_Y \sup _y c(x,y)\ d\mu <+\infty$. \\
If  $c(x,y)$ is a normal cost, i.e. a Borel measurable map such that $c(x,\cdot )$ is lower semicontinuous, then it can be reduced to a Caratheodory cost by standard approximation procedures. 
For instance  (see \cite{fonseca-leoni}), we may write 
$$ c(x,y)=\sup _h a_h(x)b_h(y), \quad b_h\in \C(Y) .$$
Hence, the cost $c_j=\sup_{i\leq j} a_ib_i$ is Caratheodory and $c_j\nearrow c$. Since
$$\int_{X\times Y}c_j(x,y)d\pi_n \leq \int_{X\times Y} c(x,y) d\pi_n ,$$ passing to the limit we obtain
$$\int_{X\times Y}c_j(x,y)d\pi\leq \liminf_{n\rightarrow + \infty}\int_{X\times Y} c(x,y) d\pi_n .$$ Passing to the limit with respect to $j$ we get that  $\pi$ is optimal. By representation of weakly* l.s.c. functionals (see \cite{a-f-p}), the same reasonings apply to normal cost on $X\times \P(Y)$. 
For a related result see also \cite{solution-kant}. 

In the following  we  compare the Kantorovich problem \eqref{Monge-Kantorovich} 
with the abstract version formulated using  the transport classes. 

\begin{lemma}\label{comp-kant}
For every transport class $\Lambda\subset\Pi(\mu,\nu)$ it results
$$ MK(c,\mu,\nu)\leq MK(\tilde c, \mu, \Lambda).$$
\end{lemma}
\begin{proof}
Let $\tilde \pi= \mathcal N(x)\otimes \mu   \in \Pi(\mu, \Lambda)$. We compute
\begin{equation}\label{comp-eq}\begin{split}
\int _{X\times \P(Y)}\tilde c \ d\tilde \pi = \int_X \left ( \int_{\P(Y)} \left ( \int _Y c(x,y) d \lambda \right ) d \mathcal N(x) \right ) d\mu =
\\
\int_X \left ( \int_Y c(x,y) df(x) \right ) d\mu = \int_{X\times Y} c(x,y)\ d(f\otimes \mu),\end{split}
\end{equation}
where $f(x) = \int_{\P(Y)} \lambda \ d \mathcal N(x)$.
It remains to check that $(\pi_2)_\#(f\otimes \mu)=\nu$. By  \eqref{lambda-class}, for every $\varphi\in \C(Y)$ we have

$$ \int_X\left ( \int _Y \varphi (y) df(x) \right ) d\mu = \int_X \left ( \int_{\P(Y)} \left ( \int _Y \varphi (y) d \lambda \right ) d \mathcal N(x) \right ) d\mu = \int_{X\times \P(Y)} I_\varphi (\lambda)\ d \tilde \pi =$$ $$= \int_{\P(Y)} I_\varphi (\lambda)\ d \Lambda = \int_{\P(Y)}\left ( \int_Y  \varphi (y) \ d\lambda \right ) d\Lambda = \int_Y \varphi (y)\ d\nu .$$ 
\end{proof}
Observe that the above definition of $f(x)=\beta(\mathcal N (x))$ can be seen as a generalized barycenter map. Indeed, we have the following
\begin{lemma}\label{baric}
The generalized barycenter map 
$\beta : \P(\P(Y))\rightarrow \P(Y)$ defined by 
$$\beta (\mathcal N)=\int _{\P(Y)} \lambda \ d \mathcal N$$ is $1$-Lipschitz with respect to the Wasserstein distance.
\end{lemma}
\begin{proof}
First observe that if $\varphi \in \lip_1(Y)$ then $I_\varphi \in \lip_1(\P(Y))$. Fixed $\varphi \in \lip _1(Y)$ we get
$$ \int_Y \varphi \ d(\beta (\mathcal N_1)-\beta (\mathcal N_2)) = \int_{\P(Y)} \left ( \int _Y \varphi \ d\lambda \right ) d (\mathcal N_1-\mathcal N_2) = 
$$ $$ = \int_{\P(Y)} I_\varphi (\lambda)\ d(\mathcal N_1-\mathcal N_2) \leq W(\mathcal N_1,\mathcal N_2).$$ 
Taking the supremum with respect to $\varphi \in \lip_1(Y)$ it results 
$$ W(\beta (\mathcal N_1),\beta (\mathcal N_2))\leq W(\mathcal N_1,\mathcal N_2).$$

\end{proof}
Observe that a probability measure $\Lambda$ on $\P(Y)$ defines a transport class iff its generalized barycenter is equal to $\nu$. 
\begin{lemma}\label{equal-class}
If $\tilde c$ is normal  then there exists a transport class $\Lambda\subset\Pi(\mu,\nu)$ such that
$$ MK(c,\mu,\nu)= MK(\tilde c, \mu, \Lambda).$$
\end{lemma}
\begin{proof}
Let $\pi_n =f_n\otimes \mu \in \Pi(\mu,\nu)$ be a minimizing sequence for $MK(c,\mu,\nu)$. Set $\Lambda _n = (f_n)_\#\mu$. By passing to a subsequence we have that $\Lambda _n \rightharpoonup \Lambda$ with respect to the weak convergence of measures. Observe that
$$ \int_{\P(Y)} \left ( \int_Y \varphi \ d\lambda \right ) d\Lambda = \int_{\P(Y)} I_\varphi (\lambda )\ d\Lambda = \lim_{n\rightarrow + \infty}\int_{\P(Y)} I_\varphi (\lambda )\ d\Lambda _n =  \lim_{n\rightarrow + \infty}\int_X I_\varphi (f_n(x))\ d\mu = $$ $$ = \lim_{n\rightarrow + \infty}\int_X \left ( \int _Y  \varphi(y)  df_n(x) \right ) d\mu = \int _Y \varphi \ d\nu .$$ Therefore, $\Lambda$ defines a transport class.
Consider the transport plans $\tilde \pi_n = (I\times f_n)_\# \mu \in \Pi(\mu, \Lambda _n)$. Passing to a subsequence we may also suppose that weakly
$\tilde \pi_n \rightharpoonup \tilde \pi \in \Pi(\mu, \Lambda)$.
Since $\tilde c$ is normal we get
$$
MK(\tilde c, \mu, \Lambda)\leq \int _{X\times \P(Y)} \tilde c\  d \tilde \pi \leq \liminf_{n\rightarrow + \infty} \int _{X\times \P(Y)} \tilde c \ d \tilde \pi_n =$$ $$ = \liminf_{n\rightarrow + \infty} \int _X \tilde c(x,f_n(x)) d\mu =\liminf_{n\rightarrow + \infty} \int_X \left (\int_Y c(x,y) df_n(x) \right ) d\mu = MK(c,\mu,\nu). $$ The result follows  by Lemma \ref{comp-kant}.
\end{proof}
By the above analysis the  Kantorovich problem over $X\times \P(Y)$ is essentially equivalent to the usual Kantorovich's one. Indeed, if $\tilde c$ is normal  (for instance if the cost $c$  satisfies the conditions of Lemma~\ref{reg}) 
by Lemma~\ref{equal-class} we have $MK(c,\mu,\nu)= MK(\tilde c, \mu, \Lambda)$ for a transport class $\Lambda$. 
 If $\tilde \pi=\mathcal N(x)\otimes \mu  $ is a transport plan for $MK(\tilde c, \mu ,  \Lambda )$, for such transport class $\Lambda$, setting $f(x)= \int_{\P(Y)} \lambda \ d \mathcal N(x)$, by  \eqref{comp-eq} it follows that $f(x)\otimes \mu $ is an optimal plan for $MK(c,\mu , \nu)$.
For a related relaxation procedure see \cite{extreme}. 
\section{Existence and uniqueness for the Monge problem}
In the previous section we have seen that the Monge-Kantorovich problem is essentially equivalent to the transport problem on transport classes. Of course, the Monge problem reveals hard to handle. For linear cost $c(x,y)=\langle x,y \rangle $, which is equivalent to the quadratic cost $c(x,y)=|x-y|^2$ because of the expansion $|x-y|^2 = |x|^2 + |y|^2 -2\langle x, y \rangle$,  it is relatively easy to find existence and uniqueness of optimal transport maps. 
We expect some advantage by considering the special form of the cost $\tilde c$ for the Monge problem on a transport class. In order to handle with the abstract Monge problem, we review here some usual tools for solving the Monge problem.
The available approaches to existence and uniqueness rely more or less on two basic facts. The first one is based on the notion of  $c$-cyclical monotonicity. A set $S\subset X\times Y$ is said  $c$-cyclical monotone if for any finite set of pairs $(x_1,y_1), \ldots , (x_k, y_k)$ and any permutation $\sigma$ the following inequality holds true
$$ \sum_{i=1}^k c(x_i,y_i)\leq \sum_{i=1}^k c(x_i,y_{\sigma(i)}).$$ A fundamental fact in mass transportation is that the support of every optimal transport plan is  a $c$-cyclical monotone set and every $c$-cyclical monotone set is contained in the $c$-superdifferential $\partial ^c \psi $ (or contact set) of a $c$-concave function $\psi$,  where 
\begin{equation}\label{c-diff} 
\partial ^c\psi =\{y\in Y \ : \ \psi (x')-\psi (x)\leq c(x',y)-c(x,y)\  \forall x' \in X \} .
\end{equation}
A function $\psi$ is said to be $c$-concave if there exist $A\times B\subset Y\times \R$ such that $$\psi(x)=\inf_{(y,t)\in A\times B} c(x,y)+t.$$
For details we refer the reader for instance to \cite{pratelli, v2}. 
The $c$-transform of $\psi$ is defined by $\psi^c(y)=\inf_{x\in X}\{c(x,y)-\psi (x) \}$. It can be shown that $y\in  \partial ^c\psi\Leftrightarrow \psi(x)+\psi^c(y)=c(x,y)$. 
If one is able to show that for $\mu$ a.e. $x\in X$ the $c$-superdifferential is single valued, then every transport plan is supported on the graph of a transport map  (see \cite{a1, pratelli, v2}). Namely,  there exists a unique solution of the Monge problem. 
The same reasonings apply as well directly to the abstract problem $M(\tilde c,\mu, \Lambda)$. In other words,  if the $\tilde c$-superdifferentials is single-valued, then the Monge problem $M(\tilde c,\mu, \Lambda)$ admits a unique solution. 
In this framework the two Monge problems are essentially related. Indeed, suppose that the $\tilde c$-superdifferentials contain just one Dirac delta
and 
let $\psi$ be  a $c$-concave function. 
We set  
$$\tilde A=\{\delta _y :y\in A\}\times B \subset \P(Y)\times \R$$ and 
 $\tilde \psi (x)=\inf_{(\lambda, t)\in \tilde A\times B}\tilde c(x,\lambda)+t$. 
 It follows that $\tilde \psi $ is $\tilde c$-concave and $\tilde \psi (x)=\psi(x)$. 
 Recalling that $\tilde c(x,\delta_y)=c(x,y)$, the following implications hold true
$$ y\in \partial ^c\psi (x)\Leftrightarrow \psi(x')-\psi(x) \leq c(x',y)-c(x,y) \Leftrightarrow $$ 
$$ \tilde \psi(x')-\tilde \psi(x) \leq \tilde c(x',\delta_y)-\tilde c(x,\delta_ y)\Leftrightarrow \delta_y \in \partial ^{\tilde c}\tilde \psi (x). $$ Therefore $\partial ^c \psi (x)=\{y\}.$

Vice-versa, suppose that $c$-superdifferentials are single valued and let $\tilde \psi$ be a $\tilde c$-concave function. Consider the $c$-transform
$(\tilde \psi) ^c(y)=\inf_x\{ c(x,y)-\tilde \psi(x)\}$.\\ 
We have
$$ \delta_y \in \partial ^{\tilde c}\tilde \psi (x)\Leftrightarrow \tilde \psi(x')-\tilde \psi (x)\leq c(x',y)-c(x,y) \Leftrightarrow c(x,y)-\tilde \psi (x)\leq c(x',y)-\tilde \psi (x')$$ $$ \Leftrightarrow c(x,y)-\tilde \psi (x) = (\tilde \psi )^c(y) \Leftrightarrow y \in \partial ^c \tilde \psi (x) \subset \partial ^c u(x)$$ 
for a $c$-concave function $u$  (see for instance \cite[Remark 3.12, Theorem 3.10]{pratelli}). It follows that  $\tilde \psi ^c(y)$ contains just  one delta. 
Of course this singleton condition of superdifferentials can be achieved under  additional requirements on the cost function. 
  For a differentiable cost a general condition relies in  the so called {\em twist (or Spence-Mirrlees in economic settings) condition}, i.e. 
  \begin{equation}
  \label{twist}
  x\mapsto c(x,y_1)-c(x,y_2)\:\:  \hbox{has no critical point}\:\:  \forall \;y_1\neq y_2
  \end{equation}
    (see \cite{gangbo, levin, carlier}). For a generalization of such condition,  in the case of suitable geometries see \cite{mccann}. 
  In the case  $X=Y=M$ with $M$ a Riemannian manifold and for a Lagrangian cost it is enough for the cost $c$ to satisfy the Mather's shortening principle 
  and the connectedness of  the $c$-superdifferential,  as it is shown in \cite[Chapter 9]{v2}. Observe that in this case the connectedness is a key property. This property is easily satisfied if $c(x,\cdot )$ is linear as happens just for the cost $c(x,y)=\langle x, y\rangle$. 
  However, general forms  of the cost $c$ which guarantee the connectedness of the $c$-superdifferential are not known. In this perspective, the consideration of the linear cost $\tilde c$ could be  useful. 
  Observe that for a cost $c$ linear with respect to the second variable, the twist condition is not in general satisfied. Consider for instance 
  $c(x,y)=\langle a(x), y \rangle $ for possibly not invertible Jacobian matrix $\nabla a(x)$.
\subsection{Monge-Mather's shortening principle}
By considering the cost $\tilde c$, it may happen that the $\tilde c$-superdifferential contains many points, which actually are probability measures of $\P(Y)$, although the $c$-superdifferential is a singleton. However, we have the advantage that the $\tilde c$-superdifferential are convex sets. 
Consider $X=Y=M$ 
and a cost
$c$ 
satisfying a  shortening principle. 
We briefly sketch the reasonings of \cite[Chapter 9]{v2}. Suppose that the following conditions are satisfied
\begin{enumerate}
\item There exists  $D\subset M$ with $\mu(D)=0$ such that $D$ intersects every nontrivial continuous curve over $M$.
\item The cost $c$ satisfies a  shortening principle.
\item The superdifferential $\partial ^c \psi $ is connected.
\end{enumerate}
By assumption $(2)$ it is possible to define a function $F : \gamma^{x,y}(\frac{1}{2})\mapsto x$ with $y\in \partial ^c \psi (x)$ having as domain the mid point of geodesics over $M$. Indeed, by definition of $c$-superdifferential we get
$$ \psi(x_1)-\psi(x)\leq c(x_1,y)-c(x,y) \quad \psi(x)-\psi (x_1)\leq c(x,y_1)-c(x_1,y_1) $$
 for every $y\in \partial ^c\psi(x), y_1\in \partial ^c\psi(x_1)$. It follows
 $$ c(x_1,y_1)-c(x,y_1)\leq c(x_1,y)-c(x,y) \Rightarrow c(x_1,y_1)+c(x,y)\leq c(x_1,y)+c(x,y_1).$$
 Hence, the shortening principle implies that 
 \begin{equation}\label{shortening}d(x,x_1)\leq K d\left(\gamma^{x,y}\left (\frac{1}{2}\right ), \eta^{x_1,y_1}\left (\frac{1}{2}\right )\right ).\end{equation}
 By the inequality \eqref{shortening} it follows that $F$ is well defined and moreover it results a Lipschitz, actually also an Holder condition works as well, map. By condition $(1)$ and $(3)$ it is possible to show that the set of points on which $\partial ^c \psi$ is not single valued is of null measure. Indeed, let $y_1,y_2 \in \partial ^c \psi(x)$. By condition $(3)$ consider a continuous curve $\rho_t$ lying in $\partial ^c \psi(x)$ connecting $y_1,y_2$. Therefore it is defined the non-trivial continuous curve $m_t=\gamma^{x,\rho_t}(\frac{1}{2})$. Hence, $x=F(m_t)$. By point $(1)$ it follows that $x\in F(D)$ which is a null measure set. 
Since $c(x,y)=\tilde c(x,\delta _y)$, we would like to prove  that  the $\tilde c$-superdifferential are a.e. single valued. 
The analogue condition for $\tilde c$ superdifferential is 
$$ \tilde c(x_1, \lambda _1)+\tilde c(x,\lambda )\leq \tilde c(x_1,\lambda)+ \tilde c(x,\lambda _1).$$
By definition, the above inequality leads to
$$ MK(c,\delta_{x_1}, \lambda _1)+ MK(c,\delta_{x}, \lambda )\leq MK(c,\delta_{x_1}, \lambda )+MK(c,\delta_{x}, \lambda _1).$$
A shortening principle should allow to estimate the distance $d(x,x_1)$ by the distance between mid points of geodesics in $M$ or in $\P(M)$.
The main problem is to find a null set $D$, due to the fact that we are now dealing with the space $M\times \P(M)$. 
\subsection{Twist condition}
Assume the cost function $c$ satisfies  the twist condition \eqref{twist}. 
Under differentiability requirement, for instance on an open and bounded set $\Omega \subset \R^N$, with a uniform Lipschitz condition
\begin{equation}\label{lip}
|c(x,y)-c(x',y)|\leq K |x-x'|, \quad \forall y\in Y,
\end{equation}
this means that the map $y\mapsto \nabla _x c(x,y)$ is injective. It turns out that $c$-concave functions $\psi$ are Lipschitz on $\Omega$. 
Indeed, it suffices to compute
$$ \psi(x_2)= \inf \{c(x_2,y)+t, (y,t)\in A\times B \} =$$ $$ =\inf \{c(x_2,y)-c(x_1,y)+c(x_1,y)+t, (y,t)\in A\times B \}\leq K|x_1-x_2|+\psi(x_1).$$
Moreover,  
if $\psi$ is differentiable at $x\in \Omega$, for every $y\in \partial ^c \psi (x)$, we have  
$$ \frac{\psi(x+tv)-\psi(x)}{t}\leq \frac{c(x+tv,y)-c(x,y)}{t}.$$ Passing to the limit as $t\rightarrow 0^+$ we get 
$$\langle \nabla \psi - \nabla _x c(x,y), v \rangle \leq 0.$$ By the arbitrariness of $v$ it follows 
$\nabla \psi (x) = \nabla _x c(x,y)$ (see also \cite{carlier}). Therefore, if $\mu$ is absolutely continuous with respect to the Lebesgue measure, then the twist condition implies  that   the superdifferential $\partial ^c\psi (x)$ is a singleton for  $\mu$-a.e. $x\in \Omega$. 
The same reasoning applies as well directly for the cost $\tilde c$. Indeed, since $\nabla _x \tilde c(x,\lambda) = \int_Y \nabla _x c(x,y) \ d\lambda $ we have that $\tilde c$-superdifferentials contains at most one delta iff the cost $c$ satisfies the twist condition. \\
In the sequel  we sketch  a related  approach to show  existence and uniqueness for the Monge problem.\\ 
Assume that $c(x,y)$ is a Caratheodory function,  hence by Lemma~\ref{reg} $c$ is continuous.
Approximate the target measure $\nu$ by a finite convex combination of Dirac deltas, say by  $\nu_n\in \P(Y)$. Consider the problem $MK(c, \mu , \nu_n)$.
The optimal transport plans of this approximation problem are of course extremal points of $\Pi(\mu, \nu_n)$. It  can be shown that these extremal points, which are supported on the graph of a $c$-concave function $\psi _n$, are of the form $(Id\times t)_\#\mu$ (see section $4$), i.e. the Monge problem $M(c,\mu,\nu_n)$ admits a unique solution. 

By the uniform Lipschitz condition \eqref{lip} it follows that the sequence of $c$-concave  functions $\psi_n$ is equi-Lipschitz. Indeed, fixed $x,x'\in \Omega$, for $y_n\in \partial_c \psi_n(x)$ we have
$$ \psi_n(x')-\psi_n(x)\leq c(x',y_n)-c(x,y_n)\leq K|x-x'|,$$ while for $y'_n\in \partial_c \psi_n(x')$ we have
$$ \psi_n(x)-\psi_n(x')\leq c(x,y'_n)-c(x',y'_n)\leq K|x-x'|.$$
By the Ascoli-Arzel\'{a} Theorem we may suppose that $\psi_n \rightarrow \psi $,  uniformly on compact subsets. Observe that $\psi$ is $c$-concave as well. 
Indeed, recall that a map $\psi$ is $c$-concave iff $\psi=\psi^{cc}$, where 
$$\psi^{cc}(x)=\inf\{ \psi (x)\leq f(x) : \ f \:\mbox{c-concave} \}.$$ Fixed $\eps >0$, we find a large integer $ n$ such that $|\psi_n(x)- \psi (x)|<\eps$. Since $\psi_n$ is $c$-concave we get 
$$ \psi (x) < \eps + \psi_n(x)\Rightarrow \psi^{cc}(x)\leq \psi_n(x)+\eps \leq \psi(x)+2\eps .$$  By the arbitrariness of $\eps$ we obtain $ \psi ^{cc}\leq \psi$. 

Let $y_n=t_n(x)\in \partial _c \psi_n(x)$. By passing to a subsequence we may suppose that $y_n\rightarrow y$. Since
$$   \psi_n(x')-\psi_n(x)\leq c(x',y_n)-c(x,y_n)\leq K|x-x'|,$$ passing to the limit as $n\rightarrow + \infty$ we get
$$ \psi(x')-\psi(x)\leq c(x',y)-c(x,y).$$ Therefore $y\in \partial _c\psi (x)$. Since the $c$-superdifferential is a singleton, the whole sequence $y_n$ converges to $y$. We set  $t(x)=y$. As limit of measurable maps, $t$ is measurable as well. Moreover, it results
$$ \int _\Omega f(t(x)) d\mu = \lim_{n\rightarrow + \infty}\int _\Omega  f(t_n(x))d\mu =\lim_{n\rightarrow + \infty}\int_Y f(y)d\nu _n = \int_Yf(y) d\nu.$$ Hence $t_\#\mu=\nu$. Since the graph $(x,t(x))$ is supported on     $ \partial _c\psi (x)$ it follows, see \cite[Th. 3.22]{pratelli},  that $t$ is an optimal transport map.  
For a related approximation procedure of the Kantorovich problem see \cite{solution-kant}.\\
Hence, under the twist condition for the cost $\tilde c$, we would have existence and uniqueness of optimal transport plans in every fixed transport class.
Therefore, an interesting question is that of finding condition on the cost $c$ ensuring the twist condition for $\tilde c$. 
\section{Existence in discrete  transport classes}
To treat discrete measures we sketch existence for the Monge problem in this setting. For an extensive  discussion of this case we refer to \cite{abdellaoui, cuesta-albertos, location}. 

Let $\nu =\sum_i a_i \delta_{y_i}$ be a discrete probability measure over a metric space $Y$. Let $\gamma$ be an optimal transport plan between $\mu$ and $\nu$. Denote by $\Gamma = {\rm supp} (\gamma)$. By optimality, we have that $\Gamma$ is a $c$-cyclically monotone set of $X\times Y$. Define $A_{i,j}=\{x\in X \ :\ (x,y_i),(x,y_j) \in \Gamma \}$. Let $x,x'\in A_{i,j}$. By $c$-cyclically monotonicity we get
$$ c(x,y_i)+c(x',y_j)\leq c(x,y_j)+c(x',y_i)\leq c(x,y_i)+c(x',y_j).$$
Therefore, we obtain
$$c(x,y_i)-c(x,y_j)=c(x',y_i)-c(x',y_j)=\lambda_{i,j}.$$ 
Under the assumption of $\mu$  $c$-continuous, i.e. $$\mu \left (\{x\in X \ : \ c(x,y_i)-c(x,y_j)=\lambda_{i,j}\}\right )=0,$$ for every $y_i,y_j \in Y, \lambda_{i,j}\in \R$, we have that $\mu(A_{i,j})=0$. Therefore, since the set of splitting masses is given by 
$A=\bigcup _{i,j} A_{i,j}$, if $\mu$  is $c$-continuous then $\mu (A)=0$. Hence, the transport plan $\gamma$ is induced by a transport map. 

\subsection{Existence in some transport classes} Let $\Lambda$ be an atomic transport class. Optimal transport plan for $MK(\tilde c, \mu,\Lambda )$ are of the form
$\sum _{i} \alpha_i\delta_{f_i} \otimes \mu$ (see \cite{extreme, solution-kant}).  For every index it results $f_i(x)\in \partial ^{\tilde c} \psi (x)$. By linearity of $\tilde c$ it follows $f(x)=\sum \alpha_i f_i \in \partial ^{\tilde c} \psi (x)$. Therefore, the transport map $f(x)$ is optimal for 
$M(\tilde c, \mu, (f)_ \# \mu)$. 

\subsection{Non-existence in some transport classes}
Consider a discrete transport class given by $\Lambda = \sum _i a_i \delta_{\lambda _i}$. Consider the cost $c(x,y)=\langle x, y \rangle$.  For $i,j$ observe that 
$$ \tilde c(x,\lambda_i)-\tilde c(x,\lambda _j)=\langle x,  \int _Y y\ d(\lambda_i-\lambda _j ) \rangle .$$
Therefore, $M(\tilde c, \mu, \Lambda)$ admits solution iff  $\beta (\lambda _i)\neq \beta (\lambda _j)$. 
Analogously, for cost $a(x)b(y)$, if $\mu (\{x \in X \  : \ a(x)=k \})=0$ for every $k\in \R$, it turns out that $M(\tilde c, \mu, \Lambda)$ admits solution iff $\int _Y b(y) d(\lambda_i- \lambda _j) \neq 0$ for every $i,j$. 

In this section we have seen that the usual approaches to solve the Monge problem give rise to some difficulties in the setting of the transport problem in a transport class. In some sense, these methods are specific for the transport class corresponding to transport maps. The question to establish existence in different transport classes remains open. We have shown that also for the case of a discrete transport class the answer could be negative. Therefore, from this point of view, the Monge problem corresponds to a \emph{lucky} case for the fixed transport class. This feature of transport classes also naturally leads to the following question. 
The existence results for the Monge problem are usually stated in the following form: under some assumption on the spaces, on the first marginal $\mu$ and on the cost $c(x,y)$,  for every second marginal $\nu$ the Monge problem admits solutions. 
Having in mind the abstract Monge problem  $ M(\tilde c, \mu, \Lambda)$,  
it could be also interesting to consider, under some assumption on the spaces, on the first marginal $\mu$ and on the cost $c(x,y)$,  the question for what kind of second marginals the corresponding Monge problem admits solutions. 

\end{document}